\documentclass[11pt,a4paper]{article}

\usepackage[english]{babel}
\usepackage[utf8]{inputenc}

\usepackage[DIV=11]{typearea}

\usepackage[affil-it]{authblk}

\usepackage{xparse}
\usepackage{amsmath, amssymb, amsfonts, amsthm}
\usepackage{mathtools}
\usepackage{enumitem}

\usepackage[noadjust]{cite}

\newcommand{\R}{\mathbb{R}}

\newcommand{\ee}{\mathrm{e}}

\DeclareDocumentCommand\dd{ o g d() }{
  \IfNoValueTF{#2}{
    \IfNoValueTF{#3}
    {\mathrm{d}\IfNoValueTF{#1}{}{^{#1}}}
    {\mathinner{\mathrm{d}\IfNoValueTF{#1}{}{^{#1}}\argopen(#3\argclose)}}
  }
  {\mathinner{\mathrm{d}\IfNoValueTF{#1}{}{^{#1}}#2} \IfNoValueTF{#3}{}{(#3)}}
}

\newcommand{\eps}{\varepsilon}

\DeclareMathOperator{\sgn}{sgn}

 \newcommand{\LL}{\mathcal{L}}

\newcommand{\vcc}{\vcentcolon}

\DeclarePairedDelimiter\abs{\lvert}{\rvert}
\DeclarePairedDelimiter\norm{\Vert}{\rVert}

\theoremstyle{plain}
\newtheorem{theorem}{Theorem}[section]

\newtheorem{proposition}[theorem]{Proposition}

\theoremstyle{definition}
\newtheorem{definition}[theorem]{Definition}

\theoremstyle{remark}
\newtheorem{remark}[theorem]{Remark}

\numberwithin{equation}{section}

\title{Convergence to equilibrium for a class of coagulation-fragmentation equations without detailed balance}
\author{Apratim Bhattacharya\thanks{Ume{\aa} University,  Department of Mathematics and Mathematical Statistics, 90187 Ume{\aa}, Sweden; National Institute of Science Education and Research, Homi Bhabha National Institute, 752050 Jatni, India,
\texttt{apr.bhattacharya@gmail.com}} { and} Sebastian Throm\thanks{Ume{\aa} University,  Department of Mathematics and Mathematical Statistics, 90187 Ume{\aa}, Sweden, \texttt{sebastian.throm@umu.se}}}
\date{}

\begin{document}

\maketitle

\begin{abstract}
 We prove convergence to equilibrium for a class of coagulation-fragmentation equations that do not satisfy a detailed balance condition. More precisely, we consider perturbations of constant rate kernels. Our result provides in particular explicit convergence rates. Uniqueness of the stationary states is proven as well for the considered class of kernels.
\end{abstract}

\section{Introduction}

Systems which are characterised by the coalescence and splitting of clusters can be observed in many natural and industrial processes. Well-known examples include cloud physics, polymerisation, algal growth or animal group size statistics for swarming species \cite{AcF97,Niw98,Fri00,Dra72,Zif85}. A prominent model to capture this behaviour is given by the \emph{coagulation-fragmentation} equation which reads
\begin{multline}\label{eq:coag:frag}
 \partial_t f_t(x)=\frac{1}{2}\int_{0}^{x}K(x-y,y)f_t(x-y)f_t(y)\dd{y}-f_t(x)\int_{0}^{\infty}K(x,y)f_t(y)\dd{y}\\*
-\frac{1}{2}\int_{0}^{x}F(y,x-y)\dd{y}f_t(x)+\int_{0}^{\infty}F(x,y)f_t(x+y)\dd{y}.
\end{multline}
Here $f_t(x)$ denotes the density of clusters of size $x\in(0,\infty)$ at time $t\geq 0$. The first two terms on the right-hand side correspond to coagulation, i.e.\@ merging of two clusters while the last two ones reflect (binary) fragmentation, i.e.\@ splitting of a cluster into two smaller pieces. The integral kernel $K(x,y)$ describes the rate at which clusters of sizes $x$ and $y$ merge, while $F(x,y)$ describes the rate at which a cluster of size $(x+y)$ splits into pieces $x$ and $y$. The factor $1/2$ in front of the first and third integral reflects the fact that these processes are symmetric.

In particular, to simplify the notation, we introduce the (symmetric) operators $\mathcal{C}_K$ and $\mathcal{F}_F$ corresponding to coagulation and fragmentation with kernels $K$ and $F$ respectively, i.e.\@ we have
\begin{equation}\label{eq:coag:frag:op}
 \begin{aligned}
\mathcal{C}_{K}(f,g)&\vcc=\frac{1}{2}\int_{0}^{x}K(x-y,y)f(x-y)g(y)\dd{y}\\
&\qquad \qquad \qquad-\frac{1}{2}f(x)\int_{0}^{\infty}K(x,y)g(y)\dd{y}-\frac{1}{2}g(x)\int_{0}^{\infty}K(x,y)f(y)\dd{y}\\
\mathcal{F}_{F}(f)&\vcc=-\frac{1}{2}\int_{0}^{x}F(y,x-y)\dd{y}f(x)+\int_{0}^{\infty}F(x,y)f(x+y)\dd{y}.
 \end{aligned}
\end{equation}
The coagulation equation \eqref{eq:coag:frag} can thus be rewritten as
\begin{equation*}
 \partial_tf_t=\mathcal{C}_{K}(f_t,f_t)+\mathcal{F}_{F}(f_t).
\end{equation*}
Moreover, the weak forms of $\mathcal{C}_K$ and $\mathcal{F}_F$ read
\begin{equation}\label{eq:coag:frag:op:weak}
 \begin{aligned}
\int_{0}^{\infty}\mathcal{C}_{K}(f,g)(x)\varphi(x)\dd{x}&=\frac{1}{2}\int_{0}^{\infty}\int_{0}^{\infty}K(x,y)f(x)g(y)\bigl[\varphi(x+y)-\varphi(x)-\varphi(y)\bigr]\dd{x}\dd{y}\\
\int_{0}^{\infty}\mathcal{F}_{F}(f)(x)\varphi(x)\dd{x}&=-\frac{1}{2}\int_{0}^{\infty}\int_{0}^{\infty}F(x,y)f(x+y)\bigl[\varphi(x+y)-\varphi(x)-\varphi(y)\bigr]\dd{x}\dd{y}.
 \end{aligned}
\end{equation}
As an immediate consequence one obtains for the choice $\varphi(x)=x$ that $\mathcal{C}_K$ and $\mathcal{F}_F$ preserve (at least formally) the total mass
\begin{equation*}
 M_1(t)=M_1(f_t)=\int_{0}^{\infty}xf_t(x)\dd{x},
\end{equation*}
i.e.\@ one has $M_1(t)=M_1(0)$ for all $t\geq 0$. In fact, it is well-known that this property holds rigorously provided the coagulation and fragmentation rates satisfy suitable bounds at infinity and zero, respectively. On the other hand, for sufficiently strong coagulation or fragmentation, phenomena known as \emph{gelation} and \emph{shattering} can occur. More precisely, after a finite time, mass is contained in infinitely large (gel) or small (dust) clusters. We refer to \cite{EMP02,Ban20} for more details on this. Throughout this work, we will rely on rate kernels for which mass conservation holds true.

An important question for coagulation-fragmentation models concerns the long-time behaviour. In fact, coagulation leads to a growth of clusters and reduction of the total number of clusters, while fragmentation acts in the opposite direction, i.e.\@ the number of clusters increases while their size becomes smaller. It is thus natural to expect that in the long-time limit these two processes balance and the solution $f_t$ of \eqref{eq:coag:frag} converges to a stationary state (\emph{equilibrium}). 

In fact such behaviour has been proven under specific structural assumptions on the rate kernels $K$ and $F$. More precisely, if the latter satisfy the \emph{detailed-balance condition}, i.e.\@ there exists a non-negative function $Q$ with suitable integrability such that
\begin{equation*}
 K(x,y)Q(x)Q(y)=F(x,y)Q(x+y) \qquad \text{for all } x,y\in (0,\infty).
\end{equation*}
It is straightforward to check that $Q$ is a stationary solution to \eqref{eq:coag:frag}. Moreover,
\begin{equation*}
 \mathcal{H}_Q(f)=\int_{0}^{\infty}f(x)\bigl(\ln(f(x)/Q(x))-1\bigr)\dd{x}
\end{equation*}
defines a Lyapunov function for \eqref{eq:coag:frag}. Relying on the latter one can prove that 
\begin{equation}\label{eq:db:conv}
 f_t\longrightarrow Q\qquad \text{as } t\to \infty
\end{equation}
provided $M_1(f_0)=M_1(Q)$. We note that $Q_z=z^xQ(x)$ for $z>0$ is a stationary solution to \eqref{eq:coag:frag} as well provided that $M_1(Q_z)<\infty$. We refer to \cite{LaM04} and references therein for more details. In general no rate of convergence is known for \eqref{eq:db:conv}.

However, in the special case of constant kernels, more precise results can be obtained and one can show that $f_t$ approaches the equilibrium at an exponential rate. In particular, the (unique) equilibrium is explicitly given by $Q(x)=\ee^{-x/\sqrt{\rho}}$ where $\rho$ is the total mass $\rho=M_1(f_0)$. The precise result is stated in Theorem~\ref{Thm:AB} below.

A special class of coagulation-fragmentation models, satisfying the detailed-balance condition, where explicit convergence rates have been obtained is the \emph{Becker-Döring} model. The latter is a discrete version of \eqref{eq:coag:frag} where the clusters, consisting of atomic particles, only interact with these monomers. More precisely, a  cluster can only merge with or emit a monomer. Relying again on entropy methods, explicit convergence rates have been obtained for the Becker-Döring model \cite{JaN03,CaL13}.

To our knowledge, the only result where convergence towards equilibrium without the detailed-balance condition could be shown is \cite{FoM04}. The latter result considers the discrete version of \eqref{eq:coag:frag} in the regime of small total mass.

In contrast, in this work, we will provide convergence to equilibrium for \eqref{eq:coag:frag} for a class of rate kernels without satisfying a detailed-balance condition and for arbitrary mass. For this we rely on an idea developed in the context of the Boltzmann equation \cite{MiM09}. More precisely, we consider kernels $K_\eps=2+\eps W$ and $F=2+\eps V$ for small $\eps>0$ and follow a perturbative approach relying on the result of \cite{AiB79}. The procedure is similar to proving convergence to self-similarity for the pure coagulation equation in \cite{CaT21}.

\subsection{Notation and Definitions}

For $\alpha\geq 0$ we denote by $L^1_\alpha$ the weighted Lebesgue space with weight $(1+x)^{\alpha}$ on $(0,\infty)$, i.e.\@ $L_\alpha^1=L^1((1+x)^{\alpha}\dd{x})=\{f\in L^1(0,\infty)\,|\,\norm{f}_{L_\alpha^1}=\int_{0}^{\infty}(1+x)^{\alpha}\abs{f(x)}\dd{x}<\infty\}$. Moreover, for $\beta\in \R$ we denote by $M_\beta(g)$ the moment of order $\beta$ of a function $g$, i.e.\@ $M_\beta(g)=\int_{0}^{\infty}x^\beta g(x)\dd{x}$, provided the integral exists.

\begin{definition}\label{Def:mild:sol}
 For non-negative $f_0\in L_\alpha^1$ with $\alpha\geq 1$ we denote $f\in C([0,\infty),L_\alpha^1)$ a (mild) solution to \eqref{eq:coag:frag} if $f_t$ is non-negative for all $t\geq 0$ and 
 \begin{equation}\label{eq:mild:sol}
  f_t(\cdot)=f_0+\int_{0}^{t}\mathcal{C}_{K}(f_s,f_s)(\cdot)+\mathcal{F}_{F}(f_s)(\cdot)\dd{s}
 \end{equation}
where $f_s=f(s,\cdot)$.
\end{definition}
\begin{remark}
 It is well-known \cite{LaM02,EMR05} that mild solutions as given in Definition~\eqref{Def:mild:sol} are equivalent to weak solutions, i.e.\@
 \begin{equation*}
   \int_{0}^{\infty}\int_{0}^{\infty}f_t(x)\partial_t\varphi(t,x)-\bigl(\mathcal{C}_K(f_t,f_t)(x)-\mathcal{F}_{F}(f_t)(x)\bigr)\varphi(t,x)\dd{x}\dd{t}=\int_{0}^{\infty}f_0(x)\varphi(0,x)\dd{x}
 \end{equation*}
for all $\varphi\in C_c^1([0,\infty)\times (0,\infty))$ and renormalised solutions, i.e.\@
 \begin{equation}\label{eq:ren:sol}
   \frac{\dd}{\dd{t}}\int_{0}^{\infty}\beta(f_t)\varphi(x)\dd{x}=\int_{0}^{\infty}\bigl(\mathcal{C}_{K}(f_t,f_t)(x)+\mathcal{F}_{F}(f_t)(x)\bigr)\beta'(f_t)\varphi(x)\dd{x}
 \end{equation}
such that $f(0,\cdot)=f_0$ for all $\beta\in C^1(\R)\cap W^{1,\infty}(\R)$ and $\varphi\in L^{\infty}(0,\infty)$. In particular, choosing $\beta(x)=x$ we get
 \begin{equation}\label{eq:coag:frag:weak}
   \frac{\dd}{\dd{t}}\int_{0}^{\infty}f_t(x)\varphi(x)\dd{x}=\frac{1}{2}\int_{0}^{\infty}\bigl(K(x,y)f_t(x)f_t(y)-F(x,y)f_t(x+y)\bigr)\bigl[\varphi(x+y)-\varphi(x)-\varphi(y)\bigr]\dd{x}.
 \end{equation}
\end{remark}
We want to study situations where $K$ and $F$ do not necessarily satisfy a \emph{detailed-balance} relation. More precisely, we will consider perturbations of constant kernels $K_0=2$ and $F_0=2$ as follows
\begin{equation}\label{eq:ass:kernels}
 \begin{gathered}
   K_{\eps}(x,y)=2+\eps W(x,y)\qquad \text{and}\qquad F_{\eps}(x,y)=2+\eps V(x,y)\\
   \text{such that }\qquad 0\leq V(x,y)\leq \frac{1}{x+y} \quad \text{and} \quad 0\leq W(x,y)\leq 1.
 \end{gathered}
\end{equation}

\subsection{Operator bounds}

We collect some basic estimates on the operators $\mathcal{C}_{K}$ and $\mathcal{F}_{F}$. For bounded rate kernels $K$ we recall from \cite[Proposition~2.5]{CaT21} that $\mathcal{C}_K$ is a bounded bilinear form in the following sense.

\begin{proposition}\label{Prop:bound:CK}
 Let $K\colon (0,\infty)^2\to (0,\infty)$ be bounded. The operator $\mathcal{C}_K$ defined in \eqref{eq:coag:frag:op} is a continuous bilinear operator on $L_\alpha^1$ for each $\alpha\geq 0$ satisfying
 \begin{equation*}
  \norm{\mathcal{C}_{K}(f,g)}_{L^1_\alpha}\leq \frac{3}{2}\norm{K}_{L^\infty}\norm{f}_{L^1_\alpha}\norm{g}_{L^1_\alpha}
 \end{equation*}
  for all $f,g\in L^1_\alpha$.
\end{proposition}
The following proposition provides an analogous statement for the fragmentation operator.
\begin{proposition}\label{Prop:bound:FV}
 Let $V$ satisfy $0\leq V(x,y)\leq \frac{\nu_*}{x+y}$. Then $\mathcal{F}_{V}$ defined in \eqref{eq:coag:frag:op} is a linear bounded operator on $L_\alpha^1$ for each $\alpha\geq 0$, i.e.\@ there exists a constant $C_V>0$ such that
  \begin{equation*}
  \norm{\mathcal{F}_{V}(f)}_{L^1_\alpha}\leq \frac{3\nu_*}{2}\norm{f}_{L^1_\alpha}
 \end{equation*}
 for all $f\in L^1_\alpha$.
\end{proposition}

\begin{proof}
 From the definition of $\mathcal{F}_{V}$ we immediately estimate
 \begin{multline*}
  \norm{\mathcal{F}_{F}(f)}_{L_\alpha^1}\leq \frac{1}{2}\int_{0}^{\infty}(1+x)^{\alpha}\int_{0}^{x}F(y,x-y)\dd{y}|f(x)|\dd{x}\\*
  +\int_{0}^{\infty}\int_{0}^{\infty}(1+x)^{\alpha} F(x,y)|f(x+y)|\dd{x}\dd{y}.
 \end{multline*}
 Rewriting the right-hand side by means of Fubini's theorem we can further estimate 
 \begin{multline*}
  \norm{\mathcal{F}_{F}(f)}_{L_\alpha^1}\leq \frac{\nu_*}{2}\int_{0}^{\infty}(1+x)^{\alpha}\int_{0}^{x}\frac{1}{x}\dd{y}|f(x)|\dd{x}+\nu_*\int_{0}^{\infty}\int_{0}^{x}\frac{(1+x-y)^{\alpha}}{x}|f(x)|\dd{x}\dd{y}\\*
  \leq \frac{\nu_*}{2}\int_{0}^{\infty}(1+x)^{\alpha}|f(x)|\dd{x}+\nu_*\int_{0}^{\infty}(1+x)^{\alpha}{x}|f(x)|\dd{y}\dd{x}.
 \end{multline*}
In the last step we also used that $(1+x-y)^\alpha\leq (1+x)^{\alpha}$ since $\alpha\geq0$. This then shows $\norm{\mathcal{F}_{F}(f)}_{L_\alpha^1}\leq\frac{3\nu_*}{2}\norm{f}_{L_\alpha^1}$.

\end{proof}

\subsection{Well-posedness of \eqref{eq:coag:frag} and stationary solutions}

There is a robust procedure available to prove the existence and uniqueness of solutions to \eqref{eq:coag:frag} which relies on weak-compactness arguments in $L^1$ spaces. In fact we have the following result.

\begin{theorem}\label{Thm:ex:sol}
 Let $K$ and $F$ satisfy $a_*\leq K(x,y)\leq A_*$ and $b_*\leq F(x,y)\leq B_*(1+\frac{1}{x+y})$ with constants $a_*,A_*,b_*,B_*>0$ and let $\alpha\geq 2$ and $f_0\in L_\alpha^1$. For each $\rho>0$ there exists a unique solution $f_t\in C([0,\infty),L^1_\alpha)$ to \eqref{eq:coag:frag}.
\end{theorem}

The proof follows immediately by adapting the arguments in e.g.\@ \cite{EMR05}. More precisely, the proof is essentially based on uniform moment estimates which we provide later in Section~\ref{Sec:moment:est}. Since the procedure is well established, we omit the proof and refer to \cite{EMR05} or \cite{Ban20} and the references therein. 

Moreover under suitable assumptions on $K$ and $F$, \eqref{eq:coag:frag} admits an equilibrium $Q$, i.e.\@ a stationary solution. In fact we have the following statement:

\begin{theorem}\label{Thm:ex:equil}
 Let $K$ and $F$ satisfy $a_*\leq K(x,y)\leq A_*$ and $b_*\leq F(x,y)\leq B_*(1+\frac{1}{x+y})$ with constants $a_*,A_*,b_*,B_*>0$. For each $\rho>0$ there exists a stationary solution $Q$ to \eqref{eq:coag:frag} such that $Q\in L^1_\alpha$ for all $\alpha\geq 1$ and $\int_{0}^{\infty}xQ(x)\dd{x}=\rho$.
\end{theorem}

Again, this result is a small variation of \cite[Theorem~4.1]{EMR05} and we therefore omit the proof since it follows the same lines (see also \cite{Ban20}). The general idea consists in applying a fixed-point argument \cite[Theorem~1.2]{EMR05}: more precisely, one can show that the flow generated by \eqref{eq:coag:frag} is weakly sequentially continuous and leaves a suitable convex and sequentially compact subset of $L^1_\alpha$ invariant which guarantees the existence of a fixed point, i.e.\@ a stationary solution.

\begin{remark}
 Note that Theorem~\ref{Thm:ex:equil} does not ensure that the given equilibrium is unique. However, for the specific class of kernels considered in this work, we will show that \eqref{eq:coag:frag} admits in fact only one equilibrium (see Theorem~\ref{Thm:uniqueness:equi} below).
\end{remark}

\subsection{The result for constant kernels by Aizenmann-Bak}

It is well-known that coagulation-fragmentation equations with constant kernels allow for relatively explicit computations. In particular, it is possible to precisely characterise the long-time behaviour of \eqref{eq:coag:frag} relying on an entropy functional. This analysis goes essentially back to Aizenmann and Bak \cite{AiB79} (see also \cite{Ban20}) and can be summarised as follows:

\begin{theorem}\label{Thm:AB}
 Let $\rho>0$ and $K=F=2$ and let $f_0\in L_1^1$  such that $f_0\log (f_0)\in L^1$ and $\rho=M_1(f_0)$. There exists a constant $C>0$ which depends on $\rho$ and increasingly on $\abs[\big]{\int_0^\infty f_0\log(f_0)\dd{x}}$ and $M_0(f_0)$ such that 
 \begin{equation*}
  \norm{f_t-\ee^{-\frac{x}{\sqrt{\rho}}}}_{L^1}\leq C \ee^{-\min\{M_0(f_0),\sqrt{\rho}\} t} \qquad \text{for all } t\geq 0
 \end{equation*}
 and all solutions $f_t$ of \eqref{eq:coag:frag}.
\end{theorem}
The convergence can be extended to $L_\alpha^1$ by means of a simple interpolation argument using Hölder's inequality together with uniform moment estimates (see Proposition~\ref{Prop:moments:evolution}). In fact, we have the following result.
\begin{theorem}\label{Thm:AB:weight}
 Let $\rho>0$, $K=F=2$ and $f_0$ be as in Theorem~\ref{Thm:AB}. Assume in addition that $f_0\in L_{\alpha_*}^1$ with some $\alpha_*>1$. There exists a constant $C>0$ depending increasingly on $\abs[\big]{\int_0^\infty f_0\log(f_0)\dd{x}}$, $M_0(f_0)$ and $\norm{f_0}_{L_{\alpha_*}^1}$ such that each solution $f_t$ to \eqref{eq:coag:frag} satisfies 
 \begin{equation*}
  \norm{f_t-\ee^{-\frac{x}{\sqrt{\rho}}}}_{L^1_{\alpha}}\leq C \ee^{-\frac{\alpha_*-\alpha}{\alpha_*}\min\{M_0(f_0),\sqrt{\rho}\} t} \qquad \text{for all } t\geq 0
 \end{equation*}
 for all $\alpha\in (1,\alpha_*)$.
\end{theorem}

Moreover, we have a similar result for the linearisation $\LL_0$ of \eqref{eq:coag:frag} with constant kernels around the equilibrium $Q(x)=\ee^{-x/\sqrt{\rho}}$ which is given by
\begin{equation*}
\LL_0 h=2\mathcal{C}_{2}(Q,h)+\mathcal{F}_{2}(h).
\end{equation*}
In \cite{AiB79} a spectral gap estimate for $\LL_0$ in $L^2$ is provided. Via the Laplace transform the linearised equation can be converted into an explicitly solvable (non-local) ODE which yields an explicit formula for the corresponding semi-group generated by $\LL_0$ in $L_\alpha^1$ \cite{LaS08}. In fact, relying on this formula one obtains the following result:

\begin{proposition}[Spectral gap for $\LL_0$]\label{Prop:spectral:gap:L0}
 For each $\alpha\geq 1$, the linear operator $\LL_0$ generates a strongly continuous semi-group on $L_\alpha^1\cap\{f|\int_{0}^{\infty}xf(x)\dd{x}=0\}$ such that
 \begin{equation*}
  \norm{\ee^{\LL_0 t}h}_{L_\alpha^1}\leq C \norm{h}_{L_\alpha^1}\ee^{-2\sqrt{\rho} t} \qquad \text{for all } t\geq 0.
 \end{equation*}
 In particular, $\LL_0$ is invertible on $L_\alpha^1\cap\{f|\int_{0}^{\infty}xf(x)\dd{x}=0\}$ and $\norm{\LL_0^{-1}h}_{L_\alpha^1}\leq \frac{C}{2\sqrt{\rho}}\norm{h}_{L_\alpha^1}$.
\end{proposition}

\subsection{Main results and outline}

Our first main result in this article concerns the uniqueness of equilibria for \eqref{eq:coag:frag} under the assumption \eqref{eq:ass:kernels}. More precisely we will prove the following statement.

\begin{theorem}\label{Thm:uniqueness:equi}
 Let $K_\eps$ and $F_\eps$ satisfy \eqref{eq:ass:kernels} and let $\rho>0$ be given. There exists $\eps_*>0$ such that for each $\eps\in(0,\eps_*)$ there exists at most one equilibrium $Q_\eps\in L_1^1$ with total mass $\int_{0}^{\infty}xQ_\eps(x)\dd{x}=\rho$ to \eqref{eq:coag:frag}.
\end{theorem}

Our second main result is then concerned with the convergence of solutions to \eqref{eq:coag:frag} to the unique equilibrium. More precisely, we will first prove the following local convergence result which states that the stationary state attracts all solutions with initial condition sufficiently close to it.

\begin{proposition}[Local convergence to equilibrium]\label{Prop:local:convergence}
 Let $K_\eps$ and $F_\eps$ satisfy \eqref{eq:ass:kernels} and let $\rho>0$ be given. For each $\alpha\geq 2$ there exist constants $\eps_*,C_*,c,\delta>0$ such that for each $\eps\in (0,\eps_*)$ and  $f_0\in L_\alpha^1$ with $\norm{f_0-Q_\eps}_{L_\alpha^1}\leq \delta$ the corresponding solution $f_t$ of \eqref{eq:coag:frag} with total mass $\rho$ satisfies
 \begin{equation*}
  \norm{f_t-Q_\eps}_{L_\alpha^1}\leq C_*\norm{f_0-Q_\eps}_{L^1_\alpha}\ee^{-(2\sqrt{\rho}-c\eps)t} \qquad \text{for all }t\geq 0.
 \end{equation*}
Here $Q_\eps$ is the stationary solution to \eqref{eq:coag:frag} with total mass $\rho$ (which is unique assuming implicitly that $\eps_*$ is small enough).
\end{proposition}

Based on the previous result and Theorem~\ref{Thm:AB:weight} we can also prove the following global convergence result which shows that the unique equilibrium attracts all solutions with initial condition in a fixed large region provided the perturbation parameter $\eps$ is sufficiently small.

\begin{theorem}[Global convergence to equilibrium]\label{Thm:global:convergence}
 Let $K_\eps$ and $F_\eps$ satisfy \eqref{eq:ass:kernels} and let $\rho>0$ be given. For each $\alpha\geq 2$ and $R>0$ there exist constants $\eps_*,C_*,c,\delta>0$ such that for each $\eps\in(0,\eps_*)$ and $f_0\in L_\alpha^1$ with $\abs[\big]{\int_0^\infty f_0\log(f_0)\dd{x}}+M_0(f_0)\leq R$ the corresponding solution $f_t$ of \eqref{eq:coag:frag} with total mass $\rho$ satisfies
 \begin{equation*}
  \norm{f_t-Q_\eps}_{L_\alpha^1}\leq C_*\norm{f_0-Q_\eps}_{L^1_\alpha}\ee^{-(2\sqrt{\rho}-c\eps)t} \qquad \text{for all }t\geq 0.
 \end{equation*}
Here $Q_\eps$ is the stationary solution to \eqref{eq:coag:frag} with total mass $\rho$ (which is unique assuming implicitly that $\eps_*$ is small enough).
\end{theorem}

The remainder of the article is organised as follows. In Section~\ref{Sec:moments} we provide uniform moment estimates both for the time-dependent problem as well as for the stationary solutions. In Section~\ref{Sec:stability} we will prove the required stability estimates for the stationary states and the flow generated by \eqref{eq:coag:frag}. Based on this, we will then provide in Section~\ref{Sec:spectral:gap} the exponential decay of the semi-group generated by the linearisation of the coagulation-fragmentation operator around an equilibrium. In Sections~\ref{Sec:uniqueness} and \ref{Sec:convergence} we will prove our main results, i.e.\@ the uniqueness of stationary states as well as that the latter attract all solutions to \eqref{eq:coag:frag} with an exponential convergence rate upon assuming suitable bounds on the initial data.

\section{Moment estimates}\label{Sec:moments}

In this section we provide uniform moment estimates on the solutions to \eqref{eq:coag:frag}. The proofs rely on arguments which are well established and have been used in various works before. In order to be self-contained, we provide the main estimates while the presentation below is close to \cite{EMR05}.

\subsection{Moment estimates for the equilibria}

The following statement provides uniform estimates on all non-negative moments for stationary states of \eqref{eq:coag:frag} in terms of the total mass $\rho$.

\begin{proposition}\label{Prop:moments:equi}
 Let $\rho>0$ and $\eps_*>0$ be fixed and let $K_\eps$ and $F_\eps$ satisfy \eqref{eq:ass:kernels} with $\eps\in (0,\eps_*)$. Let $Q_\eps$ be a stationary solution of \eqref{eq:coag:frag} such that $\int_{0}^{\infty}xQ_{\eps}(x)\dd{x}=\rho$. For any $m\in [0,\infty)$ there exists a constant $C_m=C_m(\rho,\eps_*)$ such that 
 \begin{equation*}
   \int_{0}^{\infty}x^m Q_{\eps}(x)\dd{x} \leq C_{m}.
 \end{equation*}
\end{proposition}

\begin{proof}
 For $a\in \R$ we denote by $M_a$ the moment of order $a$ of $Q_\eps$ provided it exists, i.e.\@ $M_a=\int_{0}^{\infty}x^{a}Q_{\eps}(x)\dd{x}$.  From the weak formulation of \eqref{eq:coag:frag} we notice that $Q_\eps$ satisfies
 \begin{multline}\label{eq:equi:weak}
  \frac{1}{2}\int_{0}^{\infty}\int_{0}^{\infty}F_\eps(x,y)\bigl[\varphi(x+y)-\varphi(x)-\varphi(y)\bigr]Q_\eps(x+y)\dd{x}\dd{y}\\*
  =\frac{1}{2}\int_{0}^{\infty}\int_{0}^{\infty}K_\eps(x,y)\bigl[\varphi(x+y)-\varphi(x)-\varphi(y)\bigr]Q_\eps(x)Q_\eps(y)\dd{x}\dd{y}.
 \end{multline}
 Choosing $\varphi\equiv 1$, we deduce together with \eqref{eq:ass:kernels} that 
 \begin{multline*}
  \biggl(\int_{0}^{\infty}Q_{\eps}(x)\dd{x}\biggr)^2\leq \frac{1}{2}\int_{0}^{\infty}\int_{0}^{\infty}K_{\eps}(x,y)Q_{\eps}(x)Q_{\eps}(y)\dd{x}\dd{y}\\*
  =\frac{1}{2}\int_{0}^{\infty}\int_{0}^{\infty}F_{\eps}(x,y)Q_{\eps}(x+y)\dd{x}\dd{y}\leq \frac{1}{2}\int_{0}^{\infty}\int_{0}^{x}\bigl(2+\frac{1}{x}\bigr)Q_{\eps}(x)\dd{y}\dd{x}\\*
  =\int_{0}^{\infty}xQ_\eps(x)\dd{x}+\frac{\eps}{2}\int_{0}^{\infty}Q_\eps(x)\dd{x}.
 \end{multline*}
 By means of Young's inequality we have $\frac{\eps}{2}\int_{0}^{\infty}Q_\eps(x)\dd{x}\leq \frac{\eps^2}{8}+\frac{1}{2}\bigl(\int_{0}^{\infty}Q_\eps(x)\dd{x}\bigr)^2$ such that we conclude
 \begin{equation*}
  \frac{1}{2}\biggl(\int_{0}^{\infty}Q_{\eps}(x)\dd{x}\biggr)^2\leq \biggl(\int_{0}^{\infty}xQ_{\eps}(x)\dd{x}\biggr)+\frac{\eps^2}{8}.
 \end{equation*}
 Thus, using that $\int_{0}^{\infty}xQ_{\eps}(x)\dd{x}=\rho$ we have
 \begin{equation}\label{eq:M0:upper:bd}
  M_0\leq \bigl(2\rho+\eps^2/4\bigr)^{1/2}.
 \end{equation}
To estimate higher order moments, we proceed similarly, i.e.\@ for $m>1$ we choose $\varphi_m(x)=x^m$ in \eqref{eq:equi:weak} to get
\begin{multline*}
   \frac{1}{2}\int_{0}^{\infty}\int_{0}^{\infty}F_\eps(x,y)\bigl[(x+y)^m-x^m-y^m\bigr]Q_\eps(x+y)\dd{x}\dd{y}\\*
  =\frac{1}{2}\int_{0}^{\infty}\int_{0}^{\infty}K_\eps(x,y)\bigl[(x+y)^m-x^m-y^m\bigr]Q_\eps(x)Q_\eps(y)\dd{x}\dd{y}.
\end{multline*}
Estimating by means of \eqref{eq:ass:kernels} as before and using Fubini's theorem and a change of variables on the left-hand side, we obtain
\begin{multline*}
   \int_{0}^{\infty}\int_{0}^{x}\bigl[x^m-(x-y)^m-y^m\bigr]Q_\eps(x)\dd{y}\dd{x}\\*
  \leq \frac{1}{2}\int_{0}^{\infty}\int_{0}^{\infty}(2+\eps)\bigl[(x+y)^m-x^m-y^m\bigr]Q_\eps(x)Q_\eps(y)\dd{x}\dd{y}.
\end{multline*}
We use the elementary estimate $(x+y)^m-x^m-y^m\leq C_m(x^{m-1}y+xy^{m-1})$ on the right-hand side and compute the integral in $y$ on the left-hand side explicitly to deduce
\begin{equation}\label{eq:Mm:upper:1}
   \frac{m-1}{m+1}\int_{0}^{\infty}x^{m+1}Q_\eps(x)\dd{x}\\*
  \leq (2+\eps)\int_{0}^{\infty}xQ_\eps(x)\dd{x}\int_{0}^{\infty}y^{m-1}Q_{\eps}(y)\dd{y}.
\end{equation}
By means of Hölder's inequality together with \eqref{eq:M0:upper:bd} we deduce
\begin{multline}\label{eq:Mm:upper:2}
 \int_{0}^{\infty}y^{m-1}Q_{\eps}(y)\dd{y}\leq \biggl(\int_{0}^{\infty}y^{m}Q_{\eps}(y)\dd{y}\biggr)^{\frac{m-1}{m}}\biggl(\int_{0}^{\infty}Q_{\eps}(y)\dd{y}\biggr)^{\frac{1}{m}}\\*
 \leq \bigl(2\rho+\eps^2/4\bigr)^{\frac{1}{2m}}\biggl(\int_{0}^{\infty}y^{m}Q_{\eps}(y)\dd{y}\biggr)^{\frac{m-1}{m}}
\end{multline}
as well as
\begin{multline}\label{eq:Mm:upper:3}
 \int_{0}^{\infty}x^{m}Q_{\eps}(x)\dd{x}\leq \biggl(\int_{0}^{\infty}x^{m+1}Q_{\eps}(x)\dd{x}\biggr)^{\frac{m-1}{m}}\biggl(\int_{0}^{\infty}xQ_{\eps}(x)\dd{x}\biggr)^{\frac{1}{m}}\\*
 =\rho^{\frac{1}{m}}\biggl(\int_{0}^{\infty}x^{m+1}Q_{\eps}(x)\dd{x}\biggr)^{\frac{m-1}{m}}.
\end{multline}
Together with \eqref{eq:Mm:upper:2} and \eqref{eq:Mm:upper:3} we can further estimate \eqref{eq:Mm:upper:1} from below and above to obtain
\begin{equation*}
    \frac{m-1}{m+1}\rho^{-\frac{1}{m-1}}\biggl(\int_{0}^{\infty}x^{m}Q_\eps(x)\dd{x}\biggr)^{\frac{m}{m-1}}\\*
  \leq (2+\eps)\rho\bigl(2\rho+\eps^2/4\bigr)^{\frac{1}{2m}}\biggl(\int_{0}^{\infty}y^{m}Q_{\eps}(y)\dd{y}\biggr)^{\frac{m-1}{m}}.
\end{equation*}
From this we immediately deduce
\begin{equation*}
 M_m=\int_{0}^{\infty}x^{m}Q_\eps(x)\dd{x}\leq \Bigl(\frac{m+1}{m-1}(2+\eps)\rho^{\frac{m}{m-1}}(2\rho+\eps^2/4)^{\frac{1}{2m}}\Bigr)^{\frac{m(m-1)}{2m-1}}.
\end{equation*}
\end{proof}

\subsection{Moment bounds for time-dependent solutions}\label{Sec:moment:est}
The next statement provides global moment estimates for solutions to \eqref{eq:coag:frag} in terms of the initial condition.
\begin{proposition}\label{Prop:moments:evolution}
Let $K$ and $F$ satisfy $a_*\leq K(x,y)\leq A_*$ and $b_*\leq F(x,y)\leq B_*(1+\frac{1}{x+y})$ with constants $a_*,A_*,b_*,B_*>0$ and let $m\geq 1$ and $f_0 \in L_m^1$. There exists a constant $\mu_{m}>0$ depending only on $A_*,A_*,b_*,B_*$ and the total mass $\rho=M_1(f_0)$ such that each solution $f$ to \eqref{eq:coag:frag} with initial datum $f_0$ satisfies
\begin{equation*}
M_m(f_t) \leq \max \{M_m(f_0), \mu_m\} \ \ \forall t \ \in [0, \infty). 
\end{equation*}
\end{proposition}
\begin{proof}
We first estimate $M_0(t)=M_0(f_t)$. Choosing $\varphi\equiv 1$ in \eqref{eq:coag:frag:weak} gives 
\begin{equation*}
\frac{\dd}{\dd{t}} \int_0^\infty f_t(x)\dd{x} = - \frac{1}{2} \int_0^\infty \int_0^\infty K (x,y) f_t(x) f_t(y) \dd{y} \dd{x} + \frac{1}{2} \int_0^\infty \int_0^x F(x-y,y) f_t(x) \dd{y} \dd{x}.
\end{equation*}
Hence,
\begin{equation*}
\frac{\dd}{\dd{t}} M_0(t) \leq - \frac{a_*}{2} M_0^2(t)  + \frac{B^*}{2} M_1(t)+ \frac{B^*}{2} M_0(t)= - \frac{A_*}{2} M_0^2(t) + \frac{B^*}{2} M_0(t) + \frac{B^*}{2} \rho
\end{equation*}
where we used in the last step that \eqref{eq:coag:frag} conserves the total mass, i.e.\@ $M_1(f_t)=M_1(f_0)=\rho$ for all $t\geq 0$. Integrating this differential inequality the claimed bound on $M_0$ immediately follows (see \cite[Lemma~3.3]{EMR05}).
To bound the moments of higher order we proceed in same way, i.e.\@ we choose $\varphi (x) = x^m$ in \eqref{eq:coag:frag:weak}. Using the elementary inequality 
\begin{equation*}
(x+y)^m -x^m -y^m \leq C_m (x^{m-1}y + x y^{m-1}) \ \ \text{where $m>1$, $(x,y) \in (0, \infty) \times (0, \infty)$,}
\end{equation*}
for some constant $C_m >0$ (see \cite[p. 379, Lemma 7.4.2]{Ban20}) and $(x+y)^m-x^m-y^m\geq 0$ for $m>1$ we get
\begin{multline*}
 \frac{\dd}{\dd{t}}M_m(t)\leq A_* C_m M_{1}(t)M_{m-1}(t)-\frac{b_*}{2}\int_{0}^{\infty}\int_{0}^{\infty}\bigl(1+\frac{1}{x}\bigr)\bigl[(x+y)^m-x^m-y^m]f_t(x+y)\dd{x}\dd{y}\\*
 \leq A_* C_m M_{1}(t)M_{m-1}(t)-\frac{b_*}{2}\int_{0}^{\infty}\int_{0}^{x}\bigl(1+\frac{1}{x}\bigr)\bigl[x^m-(x-y)^m-y^m]f_t(x)\dd{y}\dd{x}\\*
 \leq A_* C_m M_{1}(t)M_{m-1}(t)-\frac{b_*}{2}\int_{0}^{\infty}\int_{0}^{x}\bigl(1+\frac{1}{x}\bigr)\bigl[x^m-(x-y)^m-y^m]f_t(x)\dd{y}\dd{x}.
\end{multline*}
Since $\int_{0}^{x}\bigl[x^m-(x-y)^m-y^m]\dd{y}=\frac{m-1}{m+1}x^{m+1}$ and $M_1(t)=\rho$ for all $t\geq 0$, we obtain 
\begin{equation}\label{eq:moment:1}
\frac{\dd}{\dd{t}} M_{m}(t)  \leq A_* C_m \rho M_{m-1} (t)   - \frac{b_* }{2} \frac{m-1}{m+1}\big(M_{m+1}(t)+M_m(t)\bigr). 
\end{equation}
By means of Hölder's inequality, we get
\begin{equation}\label{eq:moment:2}
M_{m-1} (t)  = \int_0^\infty x^{m-1} f_t^{\frac{m-1}{m}}(x) f_t^{\frac{1}{m}} (x) \dd{x} \leq
M_{m}^{\frac{m-1}{m}}(t) M_{0}^{\frac{1}{m}}(t). 
\end{equation}
Similarly, we obtain
\begin{equation*}
M_{m} (t) = \int_0^\infty x ^\frac{1}{m} f_t^\frac{1}{m} (x) x ^\frac{(m-1)(m+1)}{m} f_t^\frac{m-1}{m} (x) \dd{x}  \leq M_1^\frac{1}{m}(t) M_{m+1}^{\frac{m-1}{m}}(t) . 
\end{equation*}
Thus, since $M_1(t)=\rho$ for all $t\geq 0$, we get
\begin{equation}\label{eq:moment:3}
- M_{m+1}(t) \leq - \rho^{-\frac{1}{m-1}} M_{m}^{\frac{m}{m-1}}(t). 
\end{equation}
Combining \eqref{eq:moment:2} and \eqref{eq:moment:3} with \eqref{eq:moment:1} we get 
\begin{equation*}
\frac{\dd}{\dd{t}} M_{m}(t)  \leq A_* C_m \rho M_{0}^{\frac{1}{m}}(t) M_{m}^{\frac{m-1}{m}}(t) - \frac{b_* }{2} \frac{m-1}{m+1}\big(\rho^{-\frac{1}{m-1}} M_{m}^{\frac{m}{m-1}}(t)+M_m(t)\bigr). 
\end{equation*}
Using that $M_m(t)\geq 0$ and $M_0(t)\leq \max\{M_0(f_0),\mu_0\}$ we can further estimate to get
\begin{equation*}
\frac{\dd}{\dd{t}} M_{m}(t)  \leq A_* C_m \rho \max\{M_0(f_0),\mu_0\}^{\frac{1}{m}} M_{m}^{\frac{m-1}{m}}(t) - \frac{b_* }{2} \frac{m-1}{m+1}\rho^{-\frac{1}{m-1}} M_{m}^{\frac{m}{m-1}}(t). 
\end{equation*}
The claimed bound follows then again by integrating the differential inequality as above.
\end{proof}

\section{Stability estimates}\label{Sec:stability}

In this section we provide the stability estimates both on the time-dependent and stationary solutions to \eqref{eq:coag:frag}.

\subsection{Stability of time-dependent solutions}

The following statement provides stability of solutions to \eqref{eq:coag:frag} on a finite time horizon which is a straightforward consequence of Gronwall's inequality together with uniform moment estimates.

\begin{proposition}\label{prop_st_stab_sol} 
Let $K_\eps$ and $F_\eps$ satisfy \eqref{eq:ass:kernels} and let $f_t^\eps$ denote the solution to \eqref{eq:coag:frag} for $\eps\in[0,1]$ with initial condition $f_0\in L_\alpha^1$ for $\alpha\geq 1$. There exist constants $C,A>0$ which are independent of $\eps$ such that 
\begin{equation*}
\lVert f_t^\eps (t) - f_t^0 (t) \rVert_{L^1_k} \leq \eps C (\exp{(At)}-1). 
\end{equation*}
\end{proposition}

\begin{proof}
We note that 
\begin{equation*}
 \mathcal{C}_{2+\eps W}(f_s^\eps,f_s^\eps)-\mathcal{C}_{2}(f_s^0,f_s^0)=\mathcal{C}_{2}(f_s^\eps-f_s^0,f_s^\eps+f_s^0)+\eps \mathcal{C}_{W}(f_s^\eps,f_s^\eps)
\end{equation*}
as well as
\begin{equation*}
 \mathcal{F}_{2+\eps V}(f_s^\eps)-\mathcal{F}_{2}(f_s^0)=\mathcal{F}_{2}(f_s^\eps-f_s^0)+\eps \mathcal{F}_{V}(f_s^\eps).
\end{equation*}
Thus, taking the difference of \eqref{eq:ren:sol} for $\eps$ and $\eps=0$, we obtain
\begin{multline}\label{eq:stab:1}
 \frac{\dd}{\dd{t}}\int_{0}^{\infty}(f_t^\eps-f_t^0)(x)\varphi(x)\dd{x}\\*
 =\int_{0}^{\infty}\Bigl(\mathcal{C}_{2}(f_t^\eps-f_t^0,f_t^\eps+f_t^0)(x)+\mathcal{F}_{2}(f_t^\eps-f_t^0)(x)+\eps \bigl(\mathcal{F}_{V}(f_t^\eps)(x)+\mathcal{C}_{W}(f_t^\eps,f_t^\eps)(x)\bigr)\Bigr)\varphi(x)\dd{x}.
\end{multline}
Choosing $\varphi(x)=\sgn(f_t^\eps(x)-f_t^0(x))(1+x)^{\alpha}$ and noting that $\abs{\varphi(x)}\leq (1+x)^\alpha$ we deduce from Propositions~\ref{Prop:bound:CK} and \ref{Prop:bound:FV} that 
\begin{multline}\label{eq:stab:2}
 \abs[\bigg]{\int_{0}^{\infty}\Bigl(\mathcal{C}_{2}(f_t^\eps-f_t^0,f_t^\eps+f_t^0)(x)+\eps \bigl(\mathcal{F}_{V}(f_t^\eps)(x)+\mathcal{C}_{W}(f_t^\eps,f_t^\eps)(x)\bigr)\Bigr)\varphi(x)\dd{x}}\\*
 \leq \frac{3}{2}\norm{f_t^\eps+f_t^0}_{L_\alpha^1}\norm{f_t^\eps-f_t^0}_{L_\alpha^1}+\eps\bigl(\norm{f_t^\eps}_{L_\alpha^1}+\norm{f_t^\eps}_{L_\alpha^1}^2\bigr).
\end{multline}
Moreover, from the weak formulation \eqref{eq:coag:frag:weak} together with $(f_t^\eps-f_t^0)\sgn(f_t^\eps-f_t^0)=\abs{f_t^\eps-f_t^0}$ we get
\begin{multline}\label{eq:stab:3}
 \int_{0}^{\infty}\mathcal{F}_{2}(f)(x)\varphi(x)\dd{x}\\*
 \leq\int_{0}^{\infty}\int_{0}^{\infty}\abs[\big]{(f_t^\eps-f_t^0)(x+y)}\bigl[-(1+x+y)^\alpha+(1+x)^\alpha+(1+y)^\alpha\bigr]\dd{x}\dd{y}.
\end{multline}
To estimate the right-hand side further we note that for $\alpha>1$ it holds
\begin{multline*}
(1+x)^\alpha -1 - (1+x+y)^\alpha + (1+y)^\alpha\\*
= \alpha \int_1^{1+x}  z^{\alpha-1} \dd{z} - \alpha \int_{1+y}^{1+x+y}  z^{\alpha-1}\dd{z}=\alpha \int_1^{1+x} z^{\alpha-1} - (z + y) ^{\alpha-1} \dd{z}\leq 0.
\end{multline*}
Thus, we deduce from \eqref{eq:stab:3} together with Fubini's theorem that
\begin{multline}\label{eq:stab:4}
 \int_{0}^{\infty}\mathcal{F}_{2}(f)(x)\varphi(x)\dd{x}\leq\int_{0}^{\infty}\int_{0}^{\infty}\abs[\big]{(f_t^\eps-f_t^0)(x+y)}\dd{x}\dd{y}\\*
 =\int_{0}^{\infty}\int_{0}^{x}\abs[\big]{(f_t^\eps-f_t^0)(x)}\dd{y}\dd{x}=\norm{f_t^\eps-f_t^0}_{L_1^1}\leq \norm{f_t^\eps-f_t^0}_{L_\alpha^1}.
\end{multline}
Combining this estimate with \eqref{eq:stab:1} and \eqref{eq:stab:2} and recalling $\varphi(x)=\sgn(f_t^\eps(x)-f_t^0(x))(1+x)^{\alpha}$ we get 
\begin{equation*}
 \frac{\dd}{\dd{t}}\norm{f_t^\eps-f_t^0}_{L_\alpha^1}\leq \Bigl(\frac{3}{2}\norm{f_t^\eps+f_t^0}_{L_\alpha^1}+1\Bigr)\norm{f_t^\eps-f_t^0}_{L_\alpha^1}+\eps\bigl(\norm{f_t^\eps}_{L_\alpha^1}+\norm{f_t^\eps}_{L_\alpha^1}^2\bigr).
\end{equation*}
Proposition~\ref{Prop:moments:evolution} yields that there exists a constant $C_\alpha=C_\alpha(f_0)>0$ such that
\begin{equation*}
 \frac{\dd}{\dd{t}}\norm{f_t^\eps-f_t^0}_{L_\alpha^1}\leq C_\alpha\norm{f_t^\eps-f_t^0}_{L_\alpha^1}+C_\alpha\eps.
\end{equation*}
The claim then follows from Gronwall's inequality.
\end{proof}

The next statement is a slight variation of Proposition~\ref{prop_st_stab_sol} which provides stability of solutions to \eqref{eq:coag:frag} with the same kernel but different initial condition.
\begin{proposition}\label{Prop:stability:2} 
Let $K_\eps$ and $F_\eps$ satisfy \eqref{eq:ass:kernels} and let $f_t^\eps$, $g_t^\eps$ denote solutions to \eqref{eq:coag:frag} for $\eps\in[0,1]$ with initial data $f_0,g_0\in L_\alpha^1$ for $\alpha\geq 1$. There exists a constant $B>0$ which are independent of $\eps$ such that 
\begin{equation*}
\norm{f_t^\eps (t) - g_t^\eps (t)}_{L^1_\alpha} \leq \norm{f_0^\eps-g_0^\eps}_{L^1_\alpha}\ee^{Bt}. 
\end{equation*}
\end{proposition}

\begin{proof}
 The proof follows essentially the same lines of the proof of Proposition~\ref{prop_st_stab_sol} noting that 
 \begin{equation*}
  \frac{\dd}{\dd{t}}(f_t^\eps-g_t^\eps)=\mathcal{C}_{K_\eps}(f_t^\eps-g_t^\eps,f_t^\eps+g_t^\eps)-\mathcal{F}_{2}(f_t^\eps-g_t^\eps)-\eps\mathcal{F}_{V}(f_t^\eps-g_t^\eps).
 \end{equation*}
Arguing as in \eqref{eq:stab:4} and recalling Propositions \ref{Prop:bound:CK} and \ref{Prop:bound:FV} we get
 \begin{equation*}
  \frac{\dd}{\dd{t}}\norm{f_t^\eps-g_t^\eps}_{L_\alpha^1}\leq B\norm{f_t^\eps-g_t^\eps}_{L_\alpha^1}
 \end{equation*}
 for a constant $B>0$. Gronwall's inequality thus concludes the proof.
\end{proof}

\subsection{Stability of equilibria}

In this section, following ideas from \cite{CaT21}, we provide stability estimates on equilibria of \eqref{eq:coag:frag}. More precisely, we will show that the distance of equilibria to \eqref{eq:coag:frag} for $\eps>0$ and $\eps=0$ is of order $\eps$. As a preliminary step we prove the following statement which provides stability up to order $\eps^\gamma$ for some $0<\gamma<1$.

\begin{proposition}\label{Prop:stab:equilibria}
There exist positive numbers $C_*, \gamma$ such that 
\begin{equation*}
\lVert Q_\epsilon  - Q_0  \rVert_{L^1_\alpha} \leq C_* \epsilon^ \gamma 
\end{equation*}
for each stationary solution $Q_\eps$ to \eqref{eq:coag:frag} with $\eps\geq 0$ while $Q_0=\ee^{-\frac{x}{\sqrt{\rho}}}$.
\end{proposition}
\begin{proof}
Consider the $f_t$ satisfying 
\begin{equation*}
\partial_t f_t = \mathcal{C}_2 (f_t, f_t) + \mathcal{F}_2 (f_t) \ \ \text{with} \ f_0 (0) = Q_\epsilon. 
\end{equation*}
Since $Q_\eps$ is a stationary solution, we know that 
\begin{equation*}
\partial_t Q_\epsilon = \mathcal{C}_{2+\epsilon W} (Q_\epsilon , Q_\epsilon ) + \mathcal{F}_{2+\epsilon V} (Q_\epsilon ) \ \ \text{with} \ Q_\epsilon (0) = Q_\epsilon. 
\end{equation*}
Therefore, by Proposition \ref{prop_st_stab_sol}, we have 
\begin{equation*}
\lVert Q_\epsilon - f_t \rVert_{L^1_\alpha} \leq \epsilon C  \exp{(Ct)}. 
\end{equation*}
Again, by the exponential convergence to the equilibrium for the case of constant kernel in Theorem~\ref{Thm:AB:weight}, we get that 
\begin{equation*}
\lVert f_t - Q_0  \rVert_{L^1_\alpha} \leq C \exp{(- \mu t)} \ \ \text{for a $\mu >0$.}
\end{equation*}
Consequently, we obtain 
\begin{equation*}
\lVert Q_\epsilon  - Q_0  \rVert_{L^1_\alpha} \leq \lVert Q_\epsilon - f_t \rVert_{L^1_\alpha} + \lVert f_t - Q_0  \rVert_{L^1_\alpha} \\
\leq  \epsilon C  \exp{(Ct)}  +  C \exp{(- \mu t)} \ \  \forall t \in [0,T].
\end{equation*}
To get the optimal bound, we minimise $P (t)  = \epsilon C  \exp{(Ct)}  +  C \exp{(- \mu t)}$. Clearly, $P^\prime (t)$ vanishes at $t =\frac{1}{C + \mu} \log \frac{\mu}{\epsilon C}$. Hence, $P$ attains its minimum at $t =\frac{1}{C + \mu} \log \frac{\mu}{\epsilon C}$. A straight forward calculation yields that 
\begin{equation*}
P \left( \frac{1}{C + \mu} \log \frac{\mu}{\epsilon C}\right) = 2 C \left ( \frac{\mu}{C}\right)^{\frac{C}{C + \mu}} \epsilon ^{\frac{\mu}{C + \mu}} . 
\end{equation*}
\end{proof} 

Based on Proposition~\ref{Prop:stab:equilibria} we can prove the following improved stability result providing at most linear growth of the error with respect to $\eps$.

\begin{proposition}\label{Prop:improved:stability:equilibria}
 Let $K_\eps$ and $F_\eps$ satisfy \eqref{eq:ass:kernels}. For each $\alpha\geq 1$ there exists a constant $C>0$ such that $\norm{Q_\eps-Q_0}_{L^1_\alpha}\leq C\eps$ for each equilibrium $Q_\eps$ of \eqref{eq:coag:frag} with total mass $\rho$ and all $\eps\geq 0$  while $Q_0=\ee^{-x/\sqrt{\rho}}$.
\end{proposition}
\begin{proof}
Note that $\gamma\leq 1$ in Proposition~\ref{Prop:stab:equilibria} and thus it suffices to consider $\eps\leq 1$. Since $Q_\eps$ and $Q_0$ are equilibria of \eqref{eq:coag:frag} we have
 \begin{equation*}
  0=\mathcal{C}_{2+\eps W}(Q_\eps,Q_\eps)+\mathcal{F}_{2+\eps V}(Q_\eps)-\mathcal{C}_{2}(Q_0,Q_0)-\mathcal{F}_{2}(Q_0).
 \end{equation*}
Exploiting that $\mathcal{C}_K(f,f)$ is bilinear with respect to $f$, linear in $K$ and symmetric while $\mathcal{F}_F(f)$ is linear in both $F$ and $f$, the equation can be rearranged to obtain
\begin{equation*}
 0=\mathcal{C}_{2}(Q_\eps-Q_0,Q_\eps-Q_0)+2\mathcal{C}_{2}(Q_0,Q_\eps-Q_0)+\mathcal{F}_{2}(Q_\eps-Q_0)+\eps\bigl(\mathcal{F}_{V}(Q_\eps)+\mathcal{C}_{W}(Q_\eps,Q_\eps)\bigr).
\end{equation*}
Recalling that $\LL_0 h=2\mathcal{C}_{2}(Q_0,h)+\mathcal{F}_{2}(h)$ we can further rewrite this as
\begin{equation*}
 \LL_0 (Q_\eps-Q_0)=-\mathcal{C}_{2}(Q_\eps-Q_0,Q_\eps-Q_0)-\eps\bigl(\mathcal{F}_{V}(Q_\eps)+\mathcal{C}_{W}(Q_\eps,Q_\eps)\bigr).
\end{equation*}
By means of Proposition~\ref{Prop:spectral:gap:L0} together with Propositions~\ref{Prop:bound:CK} and \ref{Prop:bound:FV} we get
\begin{equation*}
 \norm{Q_\eps-Q_0}_{L_\alpha^1}\leq C\norm{Q_\eps-Q_0}_{L_\alpha^1}^2+C(\norm{Q_\eps}_{L_\alpha^1}+\norm{Q_\eps}_{L_\alpha^1})\eps.
\end{equation*}
Proposition~\ref{Prop:moments:equi} further implies
\begin{equation}\label{eq:imp:stab:equi}
 \norm{Q_\eps-Q_0}_{L_\alpha^1}\leq C\norm{Q_\eps-Q_0}_{L_\alpha^1}^2+C_\alpha\eps.
\end{equation}
From Proposition~\ref{Prop:stab:equilibria} we thus conclude
\begin{equation*}
 \norm{Q_\eps-Q_0}_{L_\alpha^1}\leq C\eps^{2\gamma}+C_\alpha\eps\leq \max\{C,C_\alpha\}\eps^{\min\{2\gamma,1\}}.
\end{equation*}
The claim follows by iteratively using this estimate in \eqref{eq:imp:stab:equi} since $\gamma>0$.
\end{proof}

\section{Spectral gap for the linearised operator $\LL_\eps$}\label{Sec:spectral:gap}

We consider the linearisation $\LL_\eps$ of \eqref{eq:coag:frag} around an equilibrium $Q_\eps$ with fixed mass $\rho>0$ which is given by
\begin{equation*}
\LL_\eps h=2\mathcal{C}_{K_\eps}(Q_\eps,h)+\mathcal{F}_{F_\eps}(h).
\end{equation*}

A main ingredient to prove that the equilibrium to \eqref{eq:coag:frag} is unique and additionally attracts the solutions will be the following spectral gap result for the perturbed linearised operator $\LL_\eps$ which is a direct consequence of the bounded perturbation theorem.

\begin{proposition}[Spectral gap for $\LL_\eps$]\label{Prop:spectral:gap:Leps}
 Let $K_\eps$ and $F_\eps$ satisfy \eqref{eq:ass:kernels}. There exists $\eps_*>0$ and $c>0$ such that for each $\eps\in (0,\eps_*)$ the linearised operator $\LL_\eps$ generates a strongly continuous semi-group on $L_\alpha^1\cap\{f|\int_{0}^{\infty}xf(x)\dd{x}=0\}$ such that 
  \begin{equation*}
  \norm{\ee^{\LL_\eps t}h}_{L_\alpha^1}\leq C \norm{h}_{L_\alpha^1}\ee^{-(2\sqrt{\rho}-c\eps) t} \qquad \text{for all } t\geq 0.
 \end{equation*}
 In particular, by the Hille-Yosida Theorem, $\LL_\eps$ is invertible on $L_\alpha^1\cap\{f|\int_{0}^{\infty}xf(x)\dd{x}=0\}$ and $\norm{\LL_\eps^{-1}h}_{L_\alpha^1}\leq \frac{C}{2\sqrt{\rho}-c\eps}\norm{h}_{L_\alpha^1}$.
\end{proposition}

\begin{proof}
 We note that 
 \begin{multline*}
  \LL_\eps h=2\mathcal{C}_{2+\eps W}(Q_\eps,h)+\mathcal{F}_{2+\eps V}(h)=2\mathcal{C}_{2}(Q_\eps,h)+\mathcal{F}_{2}(h)+\eps\bigl(2\mathcal{C}_{W}(Q_\eps,h)+\mathcal{F}_{V}(h)\bigr)\\*
  =2\mathcal{C}_{2}(Q_0,h)+\mathcal{F}_{2}(h)+2\mathcal{C}_{2}(Q_\eps-Q_0,h)+\eps\bigl(2\mathcal{C}_{W}(Q_\eps,h)+\mathcal{F}_{V}(h)\bigr)\\*
  =\LL_{0}h+2\mathcal{C}_{2}(Q_\eps-Q_0,h)+\eps\bigl(2\mathcal{C}_{W}(Q_\eps,h)+\mathcal{F}_{V}(h)\bigr).
 \end{multline*}
Due to Propositions~\ref{Prop:bound:CK}, \ref{Prop:bound:FV} and \ref{Prop:moments:equi} the map
\begin{equation*}
 h\mapsto 2\mathcal{C}_{2}(Q_\eps-Q_0,h)+\eps\bigl(2\mathcal{C}_{W}(Q_\eps,h)+\mathcal{F}_{V}(h)\bigr)
\end{equation*}
 defines a bounded operator on $L_\alpha^1$ for all $\alpha\geq 1$. Using additionally Proposition~\ref{Prop:improved:stability:equilibria} we get
 \begin{multline*}
  \norm{2\mathcal{C}_{2}(Q_\eps-Q_0,h)+\eps\bigl(2\mathcal{C}_{W}(Q_\eps,h)+\mathcal{F}_{V}(h)\bigr)}_{L_\alpha^1}\\*
  \leq C\bigl(\norm{Q_\eps-Q_0}_{L_\alpha^1}+\eps (1+\norm{Q_\eps}_{L_\alpha^1})\bigr)\norm{h}_{L_\alpha^1}\leq C\eps \norm{h}_{L_\alpha^1}.
 \end{multline*}
 The claim then follows from Proposition~\ref{Prop:spectral:gap:L0} together with the bounded perturbation theorem.
\end{proof}

\section{Uniqueness of equilibria}\label{Sec:uniqueness}

Based on the preparations above, will now give the proof of our first main statement, i.e.\@ that for sufficiently small $\eps>0$ the coagulation fragmentation equation with kernels $K_\eps$ and $F_\eps$ has at most one equilibrium (see also e.g.\@ \cite{MiM09,CaT21, Thr21,ABC25} where the same general idea has been used).

\begin{proof}[Proof of Theorem~\ref{Thm:uniqueness:equi}]
 Let $Q_{1,\eps}$ and $Q_{2,\eps}$ be both stationary solutions to \eqref{eq:coag:frag} with total mass $\rho$. We then have
 \begin{multline}\label{eq:unique:1}
  0=\mathcal{C}_{K_\eps}(Q_{1,\eps},Q_{1,\eps})+\mathcal{F_\eps}(Q_{1,\eps})-\mathcal{C}_{K_\eps}(Q_{2,\eps},Q_{2,\eps})-\mathcal{F_\eps}(Q_{2,\eps})\\*
  =\mathcal{C}_{K_\eps}(Q_{1,\eps}-Q_{2,\eps},Q_{1,\eps}-Q_{2,\eps})+2\mathcal{C}_{K_\eps}(Q_{1,\eps}-Q_{2,\eps},Q_{2,\eps})+\mathcal{F}_{F_\eps}(Q_{1,\eps}-Q_{2,\eps}).
 \end{multline}
 Thus, considering $\LL_{2,\eps}=2\mathcal{C}_{K_\eps}(h,Q_{2,\eps})+\mathcal{F}_{F_\eps}(h)$, i.e.\@ the linearisation around $Q_{2,\eps}$ we get
 \begin{equation*}
  \LL_{2,\eps}(Q_{1,\eps}-Q_{2,\eps})=-\mathcal{C}_{K_\eps}(Q_{1,\eps}-Q_{2,\eps},Q_{1,\eps}-Q_{2,\eps}).
 \end{equation*}
 Propositions~\ref{Prop:bound:CK} and \ref{Prop:spectral:gap:Leps} thus imply for sufficiently small $\eps$ that 
 \begin{equation}\label{eq:unique:2}
  \norm{Q_{1,\eps}-Q_{2,\eps}}_{L_{\alpha}^1}\leq C\norm{Q_{1,\eps}-Q_{2,\eps}}_{L_{\alpha}^1}^2.
 \end{equation}
 Moreover, due to Proposition~\ref{Prop:improved:stability:equilibria} we know 
 \begin{equation*}
  \norm{Q_{1,\eps}-Q_{2,\eps}}_{L_{\alpha}^1}\leq \norm{Q_{1,\eps}-Q_{0}}_{L_{\alpha}^1}+\norm{Q_{0}-Q_{2,\eps}}_{L_{\alpha}^1}\leq C\eps.
 \end{equation*}
Thus, for sufficiently small $\eps>0$ the estimate \eqref{eq:unique:2} can only hold if $\norm{Q_{1,\eps}-Q_{2,\eps}}_{L_{\alpha}^1}=0$.
\end{proof}

\section{Convergence to equilibrium}\label{Sec:convergence}

Relying on the results from the previous sections, we will now provide the proofs of our main statements on the convergence towards equilibrium which follow from similar arguments used in \cite{CaT21}[Proposition 8.1, Theorem 8.2]. However, in order to be self-contained, we provide the proofs. We first give the proof of the local convergence result.

\begin{proof}[Proof of Proposition~\ref{Prop:local:convergence}]
 Since $f_t$ and $Q_\eps$ both solve \eqref{eq:coag:frag}, we can take the difference which yields in the same way as in \eqref{eq:unique:1} that
 \begin{equation*}
  \partial_{t}(f_t-Q_\eps)=\LL_{\eps}(f_t-Q_\eps)+\mathcal{C}_{K_\eps}(f_t-Q_\eps,f_t-Q_\eps).
 \end{equation*}
Duhamel's formula then implies
\begin{equation*}
 (f_t-Q_\eps)(x)=\ee^{\LL_\eps t}(f_0-Q_\eps)+\int_{0}^{t}\ee^{\LL_{\eps}(t-s)}\mathcal{C}_{K_\eps}(f_s-Q_\eps,f_s-Q_\eps)\dd{s}.
\end{equation*}
By means of Propositions \ref{Prop:bound:CK} and \ref{Prop:spectral:gap:Leps} we thus deduce
\begin{equation*}
 \norm{f_t-Q_\eps}_{L_\alpha^1}\leq C\norm{f_0-Q_\eps}\ee^{-(2\sqrt{\rho}-c\eps)t}+C\int_{0}^{t}\ee^{-(2\sqrt{\rho}-c\eps)(t-s)}\norm{f_s-Q_\eps}_{L_\alpha^1}^2\dd{s}.
\end{equation*}
Gronwall's inequality applied to the square of the previous estimate yields
\begin{equation*}
\begin{split}
 \norm{f_t-Q_\eps}_{L_\alpha^1}&\leq \Bigl(\frac{1}{C\norm{f_0-Q_\eps}}-\frac{C^2}{{2\sqrt{\rho}-c\eps}}(1-\ee^{-(2\sqrt{\rho}-c\eps)t})\Bigr)^{-1}\ee^{-(2\sqrt{\rho}-c\eps)t}\\
 &\leq \Bigl(\frac{1}{C\norm{f_0-Q_\eps}}-\frac{C^2}{{2\sqrt{\rho}-c\eps_*}}\Bigr)^{-1}\ee^{-(2\sqrt{\rho}-c\eps)t}.
 \end{split}
\end{equation*}
If $\norm{f_0-Q_\eps}_{L_\alpha^1}\leq \frac{\lambda}{2C^3}$ we have $\frac{1}{C\norm{f_0-Q_\eps}}-\frac{C^2}{{2\sqrt{\rho}-c\eps_*}}\geq \frac{1}{2C\norm{f_0-Q_\eps}}$ which thus implies
\begin{equation*}
 \norm{f_t-Q_\eps}_{L_\alpha^1}\leq 2C\norm{f_0-Q_\eps}\ee^{-(2\sqrt{\rho}-c\eps)t}.
\end{equation*}
\end{proof}

We can finally prove the global convergence result given in Theorem~\ref{Thm:global:convergence}.

\begin{proof}[Proof of Theorem~\ref{Thm:global:convergence}]
 The convergence to equilibrium for constant kernels together with the stability of the flow and the equilibria allows to reduce the proof to Proposition~\ref{Prop:local:convergence}. In fact, if $f_t^0$ denotes the solution to ~\eqref{eq:coag:frag} with constant kernels and $Q_0$ the corresponding equilibrium, we have
 \begin{equation*}
  \norm{f_t^\eps-Q_\eps}_{L_\alpha^1}\leq \norm{f_t^\eps-f_t^0}_{L_\alpha^1}+\norm{f_t^0-Q_0}_{L_\alpha^1}+\norm{Q_0-Q_\eps}_{L_\alpha^1}.
 \end{equation*}
By means of Propositions \ref{prop_st_stab_sol}, \ref{Prop:stab:equilibria} and Theorem~\ref{Thm:AB:weight} we get
\begin{equation*}
   \norm{f_t^\eps-Q_\eps}_{L_\alpha^1}\leq C_1\eps \ee^{At}+C_2\ee^{-\frac{\alpha_*-\alpha}{\alpha_*}\min\{M_0(f_0),\sqrt{\rho}\}t}+C_3\eps.
\end{equation*}
Let $\delta>0$ be given by Proposition~\ref{Prop:local:convergence}. Choosing first $t_*>0$ sufficiently large such that $C_2\ee^{-\frac{\alpha_*}{\alpha_*-\alpha}\min\{M_0(f_0),\sqrt{\rho}\}t}<\delta/3$ and then $\eps_*>0$ sufficiently small such that $C_1\eps \ee^{At_*},C_3\eps<\delta/3$ for $\eps\leq \eps_*$ we obtain from Proposition~\ref{Prop:local:convergence} that
\begin{equation*}
 \norm{f_t^\eps-Q_\eps}_{L_\alpha^1}\leq C_4\norm{f_{t_*}^\eps-Q_\eps}_{L_\alpha^1}\ee^{-(2\sqrt{\rho}-c\eps)(t-t_*)}.
\end{equation*}
Due to Proposition~\ref{Prop:stability:2} there is a constant $B>0$ such that $\norm{f_{t_*}^\eps-Q_\eps}_{L_\alpha^1}\leq \norm{f_0-Q_\eps}_{L_\alpha^1}\ee^{Bt_{*}}$ which finally implies
\begin{equation*}
 \norm{f_t^\eps-Q_\eps}_{L_\alpha^1}\leq C_4\ee^{(2\sqrt{\rho}-c\eps+B)t_*}\norm{f_0-Q_\eps}_{L_\alpha^1}\ee^{-(2\sqrt{\rho}-c\eps)t}.
\end{equation*}
\end{proof}

\subsection*{Acknowledgments}

This work has been supported by the Kempe Foundation through the grant JCSMK22-0153.

\end{document}